\documentclass[11pt,a4paper,draft]{amsart}
\usepackage[ansinew]{inputenc}
\usepackage[T1]{fontenc}
\usepackage[english]{babel}

\usepackage{enumerate}
\usepackage{amsmath}
\usepackage{amsthm}
\usepackage{amssymb}

\theoremstyle{definition}
\newtheorem{theorem}{Theorem}
\newtheorem*{theorem*}{Theorem}	

\newtheorem{lemma}[theorem]{Lemma}
\newtheorem{remark}[theorem]{Remark}
\newtheorem{definition}[theorem]{Definition}

\newtheorem{corollary}[theorem]{Corollary}

\theoremstyle{remark}

\begin{document}

\title[Spectrums of $C^{\ast }$-algebras of bounded functions]%
{Representation of spectrums of $C^{\ast }$-algebras of bounded
  functions in terms of filters}

\author[Alaste]{T. Alaste}

\address{Department of Mathematical Sciences\\
   University of Oulu\\
   PL 3000\\
   FI-90014 Oulun yliopisto\\
   Finland}

\email{tomi.alaste@gmail.com}



\maketitle

\begin{abstract}
Let $X$ be a non-empty set and let $\mathcal{F}$ be any $C^{\ast
}$-subalgebra of $\ell ^{\infty }(X)$ containing the constant functions.
We show that the spectrum of $\mathcal{F}$ can be considered as a
space of certain filters determined by $\mathcal{F}$ on $X$.
Furthermore, we show that these filters describe the topology of the
spectrum of $\mathcal{F}$.
\end{abstract}

\section{Introduction}
\label{sec:introduction}

It is well known that filters have a fundamental role in the study of
algebraic properties of the Stone-\v Cech compactification $\beta S$
of a discrete semigroup $S$.
For any non-empty set $X$, let $\ell ^{\infty }(X)$ denote the $C^{\ast
}$-algebra of all bounded, complex-valued functions on $X$.
The Stone-\v Cech compactification $\beta X$ of a discrete topological
space $X$ can be viewed as the spectrum of $\ell ^{\infty }(X)$.
On the other hand, $\beta X$ can also be viewed as the space of all
ultrafilters on $X$ (see \cite{Comfort} or \cite{HindStra}).
The latter consideration of $\beta S$ for a discrete semigroup $S$ is
the main tool in analyzing algebraic properties of $\beta S$ in
\cite{HindStra}.

Any topological compactification of a completely regular topological space $X$ or
any semigroup compactification of a Hausdorff semitopological semigroup
$S$ is determined by the spectrum of some $C^{\ast }$-subalgebra
$\mathcal{F}$ of $\ell ^{\infty }(X)$ containing the constant
functions.
The purpose of the paper is to show that the spectrum of 
\emph{any} $C^{\ast }$-subalgebra $\mathcal{F}$ of
$\ell ^{\infty }(X)$ containing the
constant functions, where $X$ is \emph{any} non-empty set,
can be considered as the space of all $\mathcal{F}$-ultrafilters on
$X$.
Independently of the $C^{\ast }$-algebra $\mathcal{F}$ in question,
our approach has a number of similarities with the consideration of
$\beta X$ of a discrete topological space $X$ as the space of all
ultrafilters on $X$.
For example, we obtain a bijective correspondence between non-empty,
closed subsets of the spectrum of $\mathcal{F}$ and
$\mathcal{F}$-filters on $X$.
Considering the importance of filters in the study of algebraic
properties of the Stone-\v Cech compactification $\beta S$ of a discrete
semigroup $S$ and the similarities between our approach,
we believe that the method
presented in the paper can serve as a valuable tool in the study of
semigroup compactifications and also of topological compactifications.
In fact, this method was already used in \cite{ufilters} to study the
smallest ideal and its closure in the
$\mathcal{LUC}$-compactification of a topological group.
Applications of our method to more general semigroup compactifications
will appear in our upcoming paper.

Although the Stone-\v Cech compactification $\beta X$ of a discrete
topological space $X$ is the most familiar compactification which may
be considered as a space of filters,
some other compactifications have also been studied in terms of filters.
As far as we are aware, the representation of spectrums of $C^{\ast
}$-algebras of bounded functions as spaces of
filters developed in the paper is the most general one and,
for many $C^{\ast }$-subalgebras of $\ell ^{\infty }(X)$,
the first one actually using filters.
If $X$ is a completely regular topological space,
then the Stone-\v Cech compactification $\beta X$ of $X$
can be considered as the space of all $z$-ultrafilters on $X$
(see \cite{GJ} or \cite{Walker}).
If $X$ is discrete, then our approach yields the usual
representation of $\beta X$ as the space of all ultrafilters on $X$,
but for non-discrete spaces our approach gives a new representation of
$\beta X$.
In \cite{near}, the uniform compactification of a uniform space $(X,\mathcal{U})$
is considered as the space of all near ultrafilters on $X$.
However, near ultrafilters need not be filters in the ordinary sense of
the word, since they need not be closed under
finite intersections.
An approach using filters is given in \cite{ufilters}.
The $\mathcal{WAP}$-compactification
of a \emph{discrete} semigroup is described in terms of filters in
\cite{wapfilt} and
any semigroup compactification of a Hausdorff, semitopological
semigroup $S$ is considered as the space of certain equivalence
classes of $z$-ultrafilters on $S$ in \cite{toot1}.

We establish the representation of the spectrum of a $C^{\ast
}$-subalgebra $\mathcal{F}$ of $\ell ^{\infty }(X)$ containing the
constant functions as the space of all $\mathcal{F}$-ultrafilters on
$X$ in Sections \ref{sec:topol-space-delta} and
\ref{sec:cont-funct-delta}.
We introduce these filters in Section \ref{sec:mathcalf-filters} and
describe some of their properties that we will use throughout the paper.
These properties show that the approach presented in the paper is a
natural extension of the consideration of $\beta X$ of a discrete
topological space $X$ as the space of all ultrafilters on $X$.
In Section \ref{sec:topol-space-delta},
we also show that
$\mathcal{F}$-filters describe the topology of the spectrum of
$\mathcal{F}$ in a similar way as filters describe the topology of
$\beta X$ for a discrete topological space $X$.
In Section \ref{sec:inclusions-etc}, we
describe some relationships between two $C^{\ast }$-subalgebras of
$\ell ^{\infty }(X)$ containing the constant functions.
In Section \ref{sec:mathc-filt-ideals}, we establish a correspondence between
$\mathcal{F}$-filters on $X$ and closed, proper ideals of $\mathcal{F}$.
In these sections, we assume no topological or algebraic structure on
the set $X$.
The last section is devoted to a study of $\mathcal{F}$-filters in
the case that $X$ is a Hausdorff, topological space and
every member of $\mathcal{F}$ is continuous.

In the paper, we present a self-contained and elementary construction
of the spectrum of a $C^{\ast }$-subalgebra $\mathcal{F}$ of
$\ell ^{\infty }(X)$ containing the constant functions.
Therefore, we wish to keep the prerequisites for the paper clear.
The required results are Urysohn's Lemma, Dini's Theorem, and the
fact that $\lvert f\rvert \in \mathcal{F}$ for every $f\in
\mathcal{F}$ (see \cite[pp. 89-90]{Alipr}).

\section{Preliminaries}
\label{sec:preliminaries}

Throughout the paper, let $X$ be a non-empty set and let
$\mathcal{F}$ be a $C^{\ast }$-subalgebra of $\ell ^{\infty }(X)$
containing the constant functions.
We introduce shortly some notation and remind the reader of
some definitions that we will use throughout the paper.

We denote the set of all positive integers by $\mathbb{N}$, that is,
$\mathbb{N}=\{1,2,3,\ldots \}$.
We denote the set of all subsets of $X$ by $\mathcal{P}(X)$ and we
define $\mathcal{P}^{\ast }(X)=\mathcal{P}(X)\setminus \{\emptyset \}$.
A \emph{filter} on $X$ is a non-empty subset $\varphi $ of
$\mathcal{P}(X)$ with the following properties:
\begin{enumerate}[(i)]
\item If $A,B\in \varphi $, then $A\cap B\in \varphi $.
\item If $A\in \varphi $ and $A\subseteq
  B\subseteq X$, then $B\in \varphi $.
\item $\emptyset \notin \varphi $.
\end{enumerate}

A \emph{filter base} on $X$ is a non-empty subset $\mathcal{B}$ of
$\mathcal{P}(X)$ such that $\emptyset \notin \mathcal{B}$ and,
for all sets $A,B\in \mathcal{B}$, there exists a set $C\in
\mathcal{B}$ such that $C\subseteq A\cap B$.
If $\mathcal{B}$ is a filter base on $X$,
then the filter $\varphi $ on $X$ \emph{generated} by $\mathcal{B}$ is
\[
\varphi =\{ A\subseteq X : \text{there exists some set } B\in
\mathcal{B} \text{ such that } B\subseteq A \}.
\]
Let $\varphi $ be a filter on $X$.
A subset $\mathcal{B}$ of $\mathcal{P}(X)$ is a \emph{filter base} for
$\varphi $ if and only if $\mathcal{B}\subseteq \varphi $ and, for
every set $A\in \varphi $, there exists a set $B\in \mathcal{B}$ such
that $B\subseteq A$.

A real-valued function $f$ on $X$ is \emph{positive}
if and only if $f(x)\geq 0$ for every $x\in X$.
If $f$ and $g$ are real-valued members of $\mathcal{F}$,
then the functions $(f\vee g)$ and $(f\wedge g)$ in $\mathcal{F}$ are
defined by
$(f\vee g)(x)=\max \{ f(x),g(x) \}$ and  $(f\wedge g)(x)=\min
\{f(x),g(x)\}$, respectively, for every $x\in X$.

Let $(X,\tau )$ be a (not necessarily Hausdorff) topological space.
For every subset $A$ of $X$, we denote
by $\text{int}_{(X,\tau )}(A)$ and $\text{cl}_{(X,\tau )}(A)$ the
interior and the closure of $A$ in $X$,
respectively, or simply by $\text{int}_X(A)$ and $\text{cl}_X(A)$ if
$\tau $ is understood.
We denote by $C(X)$ the $C^{\ast }$- subalgebra of $\ell ^{\infty }(X)$ consisting
of continuous members of $\ell ^{\infty }(X)$.
If $X$ is a locally compact Hausdorff topological space,
then the $C^{\ast }$-subalgebra $C_0(X)$ of $C(X)$ consists of those members of
$C(X)$ which vanish at infinity.

\section{$\mathcal{F}$-filters}
\label{sec:mathcalf-filters}

In this section, we introduce $\mathcal{F}$-filters and
$\mathcal{F}$-ultrafilters on $X$ and we study some of their basic
properties.
In the next two sections, we show that the spectrum of $\mathcal{F}$
can be considered as the space of all $\mathcal{F}$-ultrafilters on
$X$ and that $\mathcal{F}$-filters describe the topology of this
space.
For every $f\in \mathcal{F}$ and for every $r>0$, we define
\[
Z(f)= \{ x\in X : f(x)=0 \} \quad \text{and} \quad X(f,r)= \{ x\in X :
\lvert f(x)\rvert \leq r \}.
\]

\begin{definition}
\label{filters_Ffamily}
An \emph{$\mathcal{F}$-family} on $X$ is a non-empty subset
$\mathcal{A}$ of $\mathcal{P}^{\ast }(X)$ such that,
for every set $A\in \mathcal{A}$ with $A\neq X$, there exist a set
$B\in \mathcal{A}$ and a function $f\in \mathcal{F}$ such that $f(B)=\{0\}$ and
$f(X\setminus A)=\{1\}$.
An \emph{$\mathcal{F}$-filter} on $X$ is a filter $\varphi $ on $X$
which is also an $\mathcal{F}$-family on $X$.
\end{definition}

Of course, we may just as well assume that the function $f\in
\mathcal{F}$ in the previous definition satisfies
$f(B)=\{1\}$ and $f(X\setminus A)=\{0\}$.
Also, since $\lvert f\rvert \in \mathcal{F}$ for every $f\in \mathcal{F}$,
we may assume that $f$ is real-valued and, if necessary, that
$f(X)\subseteq [0,1]$.

There exists at least one $\mathcal{F}$-filter on $X$, namely the
filter $\varphi =\{X\}$.
If $\mathcal{F}$ contains only constant functions, then $\{X\}$
is the only $\mathcal{F}$-filter on $X$.
On the other hand, if $\mathcal{F}=\ell ^{\infty }(X)$,
then every filter $\varphi $ on $X$ is an $\mathcal{F}$-filter on
$X$.

Let $\varphi $ be an $\mathcal{F}$-filter on $X$ and let $A\in \varphi
$ with $A\neq X$.
Pick a set $B\in \varphi $ and a function $f\in \mathcal{F}$ with
$f(B)=\{0\}$ and $f(X\setminus A)=\{1\}$.
Then $B\subseteq Z(f)\subseteq A$.
Therefore, $\varphi $ has a filter base consisting of zero sets
(determined by $\mathcal{F}$) of $X$.
However, not every zero set of $X$ is in any $\mathcal{F}$-filter.
For example, let $\mathcal{F}=C(\mathbb{R})$.
Then $A=\{0\}$ is a zero set of $\mathbb{R}$ but there is no
$\mathcal{F}$-filter $\varphi $ on $\mathbb{R}$ such that $A\in \varphi $.

We shall apply the following remark frequently
without any further notice.

\begin{remark}
\label{filters_Ffamrem}
Let $\mathcal{A}$ be a non-empty subset of $\mathcal{P}^{\ast }(X)$.
Suppose that, for every set $A\in \mathcal{A}$ with $A\neq X$, there exist
a set $B\in \mathcal{A}$, real numbers $s$ and $r$ with $s<r$,
and a real-valued function $f\in \mathcal{F}$ such that $f(x)\leq s$ for every
$x\in B$ and $f(x)\geq r$ for every $x\in X\setminus A$.
Using the lattice operations $\vee $ and $\wedge $, we see that $\mathcal{A}$ is an
$\mathcal{F}$-family on $X$.
\end{remark}

Zorn's Lemma implies that every $\mathcal{F}$-filter on $X$ is
contained in some maximal (with respect to inclusion)
$\mathcal{F}$-filter on $X$.

\begin{definition}
An \emph{$\mathcal{F}$-ultrafilter} on $X$ is an $\mathcal{F}$-filter
on $X$ which is not properly contained in any other $\mathcal{F}$-filter on
$X$.
\end{definition}

Note that if $\mathcal{F}=\ell ^{\infty }(X)$, then a filter $\varphi $ on $X$ is
an $\mathcal{F}$-ultrafilter if and only if $\varphi $ is an
ultrafilter on $X$.
Also, the following simple fact about $\mathcal{F}$-ultrafilters is very
useful: If $p$ and $q$ are $\mathcal{F}$-ultrafilters on $X$, then
$p=q$ if and only if $p\subseteq q$.

\begin{definition}
Define
\[
\mathcal{F}_0 = \{f \in \mathcal{F} : X(f,r)\neq \emptyset \text{ for
  every } r>0 \}.
\]
For every non-empty subset $A$ of $X$, define
\[
\mathcal{Z}(A)=\{ f\in \mathcal{F} : f(x)=0 \text{ for every } x\in A \}.
\]
\end{definition}

The next lemma follows from Remark \ref{filters_Ffamrem}.

\begin{lemma}
\label{filters_fr-family}
If $\mathcal{F}'\subseteq \mathcal{F}_0$ is non-empty,
then $\mathcal{A}=\{ X(f,r) : f\in
\mathcal{F}', \, r>0 \}$ is an $\mathcal{F}$-family on $X$.
\end{lemma}

We will use the following lemma and its corollaries a number of times
in the paper.
Recall that a non-empty subset $\mathcal{A}$ of $\mathcal{P}(X)$ has
the \emph{finite intersection property} if and only if $\bigcap _{k=1} ^n A_k\neq
\emptyset $ whenever $A_1,\ldots ,A_n\in \mathcal{A}$ for some $n\in
\mathbb{N}$.

\begin{lemma}
\label{filters_existence}
If $\mathcal{A}$ is an $\mathcal{F}$-family on $X$ such that $\mathcal{A}$
has the finite intersection property,
then there exists an $\mathcal{F}$-ultrafilter $p$ on $X$ such that
$\mathcal{A}\subseteq p$.
\end{lemma}

\begin{proof}
We sketch the proof briefly.
Let $\varphi $ be the smallest filter on $X$ containing $\mathcal{A}$.
Let $n\in \mathbb{N}$ and let $A_1,\ldots A_n\in \mathcal{A}$
with $A_k\neq X$ for every $k\in \{1,\ldots ,n\}$.
If $k\in \{1,\ldots ,n\}$, then there exist a set $B_k\in \mathcal{A}$ and
a positive function $f_k\in \mathcal{F}$ with $f_k(B_k)=\{0\}$
and $f_k(X\setminus A_k)=\{1\}$.
Put $B=\bigcap _{k=1}^n B_k$ and $f=\sum _{k=1}^n f_k$.
Since $B\in \varphi $, $f\in \mathcal{F}$, $f(B)=\{0\}$, and $f(x)\geq
1$ for every $x\in X\setminus \bigcap _{k=1}^nA_k$,
the filter $\varphi $ is an $\mathcal{F}$-family on $X$.
\end{proof}

The next two corollaries now follow from Lemma \ref{filters_fr-family}.

\begin{corollary}
\label{filters_extend}
Let $\varphi $ be an $\mathcal{F}$-filter on $X$ and let
$f\in \mathcal{F}$.
If $X(f,r)\cap B\neq \emptyset $ for all
$r>0$ and $B\in \varphi $,
then there exists an $\mathcal{F}$-ultrafilter $p$ on $X$ with
$\varphi \cup \{ X(f,r) : r>0 \} \subseteq p$.
\end{corollary}

\begin{corollary}
\label{filters_extend2}
Let $\varphi $ be an $\mathcal{F}$-filter on $X$ and let $A\subseteq X$.
If $A\cap B\neq \emptyset $ for every $B\in \varphi $,
then there exists an $\mathcal{F}$-ultrafilter $p$ on
$X$ containing the family $\varphi \cup \{ X(f,r) : f\in
\mathcal{Z}(A), \, r>0 \}$.
\end{corollary}

If $\mathcal{F}=\ell ^{\infty }(X)$, then we may take the functions in the next
theorem to be characteristic functions of subsets of $X$.
Except for statement (\ref{filters_ultrafilter_ii}), the conclusion of the next theorem is the
same as in \cite[Theorem 3.6]{HindStra}.

\begin{theorem}
\label{filters_ultrafilter}
If $\varphi \subseteq \mathcal{P}(X)$, then the following statements
are equivalent:
\begin{enumerate}[\upshape (i)]
\item\label{filters_ultrafilter_i}
$\varphi $ is an $\mathcal{F}$-ultrafilter on $X$.
\item\label{filters_ultrafilter_ii}
$\varphi $ is an $\mathcal{F}$-filter on $X$ and, if
  $X(f,r)\notin \varphi $ for some $f\in \mathcal{F}$ and $r>0$,
  then, for every real number $t$ with $0<t<r$, there exists a set $A\in
  \varphi $ with $X(f,t)\cap A=\emptyset $.
\item\label{filters_ultrafilter_iii}
$\varphi $ is a maximal
  subset of $\mathcal{P}(X)$ such that $\varphi $ is an
  $\mathcal{F}$-family on $X$ and $\varphi $ has the finite
  intersection property.
\item\label{filters_ultrafilter_iv}
$\varphi $ is an $\mathcal{F}$-filter on $X$ and, if
  $\bigcup _{k=1}^n A_k\in \varphi $ for some $n\in
  \mathbb{N}$ and for some subsets $A_1,\ldots ,A_n$ of $X$,
  then there exists $k\in \{1,\ldots ,n\}$ such that
  $X(f_k,r)\in \varphi $ for all $f\in \mathcal{Z}(A_k)$ and $r>0$.
\item\label{filters_ultrafilter_v}
$\varphi $ is an $\mathcal{F}$-filter on $X$ and, for every
  non-empty subset $A$ of $X$ with $A\neq X$,
  either $X(f,r)\in \varphi $ for all
  $f\in \mathcal{Z}(A)$ and $r>0$, or $X(g,r)\in \varphi $
  for all $g\in \mathcal{Z}(X\setminus A)$ and $r>0$.
\end{enumerate}
\end{theorem}

\begin{proof}
(i) $\Rightarrow $ (ii)
This follows from Corollary \ref{filters_extend2} with $g=(\lvert
f\rvert -t)\vee 0$.
Note here that $g\in \mathcal{Z}(X(f,t))$ and $X(g,r-t)\subseteq X(f,r)$.

(ii) $\Rightarrow $ (iii)
This follows from the definition of an $\mathcal{F}$-family.

(iii) $\Rightarrow $ (iv)
Suppose that (\ref{filters_ultrafilter_iii}) holds.
Clearly, $X\in \varphi $, $\emptyset \notin \varphi $, and $B\in
\varphi $ whenever $A\in \varphi $ and $A\subseteq B\subseteq X$.
To see that $\varphi $ is a filter,
let $A,B\in \varphi $.
Pick sets $C,D\in \varphi $ and functions $f,g\in
\mathcal{F}$ such that $f(C)=f(D)=\{0\}$ and $f(X\setminus
A)=g(X\setminus B)=\{1\}$.
Since $C\cap D\subseteq X(\lvert f\rvert
+\lvert g\rvert ,r)$ for every $r>0$,
we have $X(\lvert f\rvert +\lvert g\rvert ,r)\in \varphi $ for every
$r>0$ by Lemma \ref{filters_fr-family}.
Since $X(\lvert f\rvert +\lvert g\rvert ,1/2)\subseteq A\cap B$,
we have $A\cap B\in \varphi $, as required.

Suppose now that $\bigcup _{k=1}^n A_k\in \varphi $ for some $n\in
\mathbb{N}$ and for some subsets $A_1,\ldots ,A_n$ of $X$.
Suppose also that, for every $k\in \{1,\ldots ,n\}$, there exist
$r_k>0$ and $f_k\in \mathcal{Z}(A_k)$
such that $X(f_k,r_k)\notin \varphi $.
If $k\in \{1,\ldots ,n\}$,
then $\mathcal{A} =\varphi \cup \{X(f_k,t) : t>0\}$ is an
$\mathcal{F}$-family on $X$ by Lemma \ref{filters_fr-family}.
By assumption,
there exist a set $B_k\in \varphi
$ and $t_k>0$ such that $B_k\cap X(f_k,t_k)=\emptyset $.
Put $B=\bigcap _{k=1}^n B_k\in \varphi $.
Then $B\cap [\bigcup _{k=1}^n Z(f_k)]=\emptyset $, a contradiction.

(iv) $\Rightarrow $ (v)
This is obvious.

(v) $\Rightarrow $ (i)
Suppose that (\ref{filters_ultrafilter_v}) holds.
Suppose also that there exists an $\mathcal{F}$-filter $\psi $ on $X$
which properly contains $\varphi $.
Pick some set $A\in \psi \setminus \varphi $.
Pick a set $B\in \psi $ and a function $f\in \mathcal{F}$ such that
$f(B)=\{0\}$ and $f(X\setminus A)=\{1\}$.
Pick a set $C\in \psi $ and a function $g\in \mathcal{F}$ such
that $g(C)=\{1\}$ and $g(X\setminus B)=\{0\}$.
Since $X(f,1/2)\subseteq A$, we have $X(f,1/2)\notin \varphi $.
Since $f\in \mathcal{Z}(B)$, we have $X(g,1/2)\in \varphi $.
But now $X(g,1/2)\cap C=\emptyset $, a contradiction.
\end{proof}

\section{The topological space $\delta X$}
\label{sec:topol-space-delta}

In this section, we define a topology on the set of all
$\mathcal{F}$-ultrafilters on $X$ and establish some of the
properties of the resulting space.
In particular, we show that the resulting space is a compact Hausdorff
space and that $\mathcal{F}$-filters describe the topology of this
space.

\begin{definition}
Define $\delta X=\{ p : p \text{ is an $\mathcal{F}$-ultrafilter on } X\}$.
For every subset $A$ of $X$, define $\widehat{A}=\{ p\in \delta X :
A\in p\}$.
For every $\mathcal{F}$-filter $\varphi $ on $X$, define
$\widehat{\varphi }=\{ p\in \delta X : \varphi \subseteq p \}$.
\end{definition}

To be precise, we should include the $C^{\ast }$-algebra $\mathcal{F}$ in
the notation above, such as $\delta _{\mathcal{F}}(X)$.
Except in Section \ref{sec:inclusions-etc}, we use only one $C^{\ast
}$-algebra $\mathcal{F}$ in the same context, so the notation chosen
above should not cause any misunderstandings.

\begin{theorem}
\label{filters_filtprop}
If $\varphi $ and $\psi $ are $\mathcal{F} $-filters on $X$,
then the following statements hold:
\begin{enumerate}[\upshape (i)]
\item\label{filters_filtprop_i}
$\widehat{\varphi }=\bigcap _{A\in \varphi }\widehat{A}$.
\item\label{filters_filtprop_ii}
$\varphi =\bigcap _{p \in \widehat{\varphi }} p $.
\item\label{filters_filtprop_iii}
$\varphi \subseteq \psi $ if and only if $\widehat{\psi
  }\subseteq \widehat{\varphi }$.
\item\label{filters_filtprop_iv}
$\varphi =\psi $ if and only if $\widehat{\varphi }=\widehat{\psi
  }$.
\end{enumerate}
\end{theorem}

\begin{proof}
(i) This is obvious.

(ii) The inclusion $\varphi \subseteq \bigcap _{p\in \widehat{\varphi
  }} p$ is obvious, and so we need only to verify the reverse
inclusion.
Suppose that $A\subseteq X$ satisfies $A\notin \varphi $.
By Corollary \ref{filters_extend2},
there exists an element $p\in \widehat{\varphi }$ such that
$\{ X(f,r) :f\in \mathcal{Z}(X\setminus A), \, r>0 \}\subseteq p$.
Now, it is enough to show that $A\notin p$.
Suppose that $A\in p$.
Pick a set $B\in p$ and a function $f\in \mathcal{F}$ such that
$f(B)=\{1\}$ and $f(X\setminus A)=\{0\}$.
Since $f\in \mathcal{Z}(X\setminus A)$, we have $X(f,1/2)\in p$.
But now $B\cap X(f,1/2)=\emptyset $, a contradiction.

(iii) Necessity is obvious and sufficiency follows from statement
(\ref{filters_filtprop_ii}).

(iv) This follows from statement (\ref{filters_filtprop_iii}).
\end{proof}

The family $\{ \widehat{A} : A\subseteq X \}$ is a base for a topology
on $\delta X$.
We define the topology of $\delta X$ to be the topology which has the
family $\{ \widehat{A} : A\subseteq X \}$ as its base.
In particular, $\{ \widehat{A} : A\in p \}$ is the neighborhood base
of a point $p\in \delta X$.
If $Y\subseteq \delta X$, then we denote $\text{cl}_{\delta X}(Y)$
by $\overline{Y}$ with one exception:
If $A\subseteq X$, then we use $\text{cl}_{\delta X}(\widehat{A})$
instead of the cumbersome notation $\overline{\widehat{A}}$.

We denote by $\tau (\mathcal{F})$ the weakest topology $\tau $ on $X$
such that every function $f$ in $\mathcal{F}$ is continuous with
respect to $\tau $.
For every subset $A$ of $X$, we denote $\text{int}_{(X,\tau
  (\mathcal{F}))}(A)$ by $A^{\circ }$.
For every element $x\in X$, we denote by
$\mathcal{N}_{\mathcal{F}}(x)$ the neighborhood filter of $x$ in $(X,\tau (\mathcal{F}))$.

We shall apply the following remark frequently without any further
notice.

\begin{remark}
\label{filters_interior}
Let $\varphi $ be an $\mathcal{F}$-filter on $X$ and let $A\in
\varphi $ with $A\neq X$.
Pick $B\in \varphi $ and $f\in \mathcal{F}$ such that
$f(B)=\{0\}$ and $f(X\setminus A)=\{1\}$.
Then $B\subseteq \{ x\in X : \lvert f(x)\rvert <1\} \subseteq A$.
Therefore, if $C$ is any subset of $X$, then $C\in \varphi $ if and only
if $C^{\circ }\in \varphi $.
\end{remark}

\begin{theorem}
\label{filters_Nxbase}
If $x\in X$, then the family
\[
\mathcal{A}_x=\{ X(f,r) : f\in \mathcal{F}, \, f(x)=0,
\text{ and } r>0 \}
\]
is a filter base on $X$.
The filter on $X$ generated by $\mathcal{A}_x$ is the filter
$\mathcal{N}_{\mathcal{F}}(x)$ and it is an
$\mathcal{F}$-ultrafilter.
\end{theorem}

\begin{proof}
If $f,g\in \mathcal{F}$ and $r>0$,
then $X(\lvert f\rvert +\lvert g\rvert ,r)\subseteq X(f,r)\cap X(g,r)$.
Therefore, $\mathcal{A}_x$ is a filter base on $X$.
Clearly, $\mathcal{A}_x$ generates the filter
$\mathcal{N}_{\mathcal{F}}(x)$,
and so $\mathcal{N}_{\mathcal{F}}(x)$ is an $\mathcal{F}$-filter
on $X$ by Lemma \ref{filters_fr-family}.
Then $\mathcal{N}_{\mathcal{F}}(x)$ is an $\mathcal{F}$-ultrafilter
on $X$ by Theorem \ref{filters_ultrafilter} (\ref{filters_ultrafilter_iv}).
\end{proof}

The following definition is reasonable by the previous theorem.

\begin{definition}
The \emph{evaluation mapping} $e:X\rightarrow \delta X$ is defined by
$e(x)=\mathcal{N}_{\mathcal{F}}(x)$.
\end{definition}

If $A\subseteq X$ and $x\in X$, then $e(x)\in \widehat{A}$ if and
only if $x\in A^{\circ }$.
Next, let $A,B\subseteq X$.
In general, $\widehat{B}\cap e(A)=\emptyset $ does not imply $B\cap
A=\emptyset $.
However, this
implication holds if $B$ is a $\tau (\mathcal{F})$-open subset of $X$.
We apply this fact repeatedly in what follows.

We collect some properties of the space $\delta X$ in the following
lemmas.

\begin{lemma}
\label{filters_closure}
Let $A\subseteq X$ and let $p\in \delta X$.
The following statements are equivalent:
\begin{enumerate}[\upshape (i)]
\item $p\in \overline{e(A)}$.
\item $A\cap B\neq \emptyset $ for every $B\in p$.
\item $X(f,r)\in p$ for every $f\in \mathcal{Z}(A)$ and for every $r>0$.
\end{enumerate}
In particular, $p\in \overline{e(A)}$ for every $A\in p$. 
\end{lemma}

\begin{proof}
(i) $\Rightarrow $ (ii)
If $A\cap B=\emptyset $ for some $B\in p$,
then $A\cap B^{\circ }=\emptyset $,
and so $e(A)\cap \widehat{B^{\circ }}=\emptyset $. 
Since $B^{\circ }\in p$, we have $p\notin \overline{e(A)}$.

(ii) $\Rightarrow $ (iii)
This follows from Corollary \ref{filters_extend2}.

(iii) $\Rightarrow $ (i)
Suppose that $p\notin \overline{e(A)}$.
Pick a $\tau (\mathcal{F})$-open subset $B$ of $X$ such
that $B\in p$ and $\widehat{B}\cap e(A)=\emptyset $.
Then $B\cap A=\emptyset $.
Pick a set $C\in p$ and a function $f\in \mathcal{F}$ such
that $f(C)=\{1\}$ and $f(X\setminus B)=\{0\}$.
Then $f\in \mathcal{Z}(A)$.
Since $X(f,1/2)\cap C=\emptyset $,
we have $X(f,1/2)\notin p$.
\end{proof}

\begin{lemma}
\label{filters_properties}
If $A,B\subseteq X$, then the following statements hold:
\begin{enumerate}[\upshape (i)]
\item\label{filters_properties_i}
$\widehat{X\setminus A}=\delta X\setminus \overline{e(A)}$.
\item\label{filters_properties_ii}
If $A$ is a $\tau (\mathcal{F})$-open subset of $X$, then
$\overline{e(A)}=\text{cl}_{\delta X}(\widehat{A})$.
\item\label{filters_properties_iii}
$\widehat{A}=\widehat{B}$ if and only if $A^{\circ }=B^{\circ }$.
\item\label{filters_properties_iv}
$\widehat{A}=\emptyset $ if and only if $A^{\circ }=\emptyset $.
\item\label{filters_properties_v}
$\widehat{A}=\delta X$ if and only if $A=X$.
\end{enumerate}
\end{lemma}

\begin{proof}
(i)
Let $p\in \widehat{X\setminus A}$.
Since $\widehat{X\setminus A}\cap e(A)=\emptyset $,
we have $p\notin \overline{e(A)}$.
On the other hand, if $p\in \delta X\setminus \overline{e(A)}$,
then there exists a $\tau (\mathcal{F})$-open subset $C$ of $X$
such that $C\in p$ and $\widehat{C}\cap e(A)=\emptyset $.
Then $C\cap A=\emptyset $,
and so $X\setminus A\in p$, as required.
 
(ii)
The inclusion $\text{cl}_{\delta X}(\widehat{A})\subseteq
\overline{e(A)}$ holds for any subset $A$ of $X$ and follows from
statement (\ref{filters_properties_i}).
Suppose now that $A$ is a $\tau (\mathcal{F})$-open subset of $X$ and
let $p\in \overline{e(A)}$.
If $B\in p$, then $\widehat{B}\cap e(A)\neq \emptyset $, so $B^{\circ }\cap
A\neq \emptyset $, and so
$\widehat{B}\cap \widehat{A}\neq \emptyset $.
Therefore, $p\in \text{cl}_{\delta X}(\widehat{A})$.

Statement (\ref{filters_properties_iii}) follows from Remark \ref{filters_interior}.
Then (\ref{filters_properties_iv}) and
(\ref{filters_properties_v}) follow from (\ref{filters_properties_iii}).
\end{proof}

\begin{lemma}
\label{filters_intersecting}
If $A,B\subseteq X$,
then $X(f,r)\cap X(g,r)\neq
\emptyset $ for all $f\in \mathcal{Z}(A)$, $g\in \mathcal{Z}(B)$, and
$r>0$ if and only if $\overline{e(A)}\cap
\overline{e(B)}\neq \emptyset $.
\end{lemma}

\begin{proof}
Necessity follows from Lemma \ref{filters_closure}.
Suppose that $X(f,r)\cap X(g,r)\neq \emptyset $ for all $f\in
\mathcal{Z}(A)$, $g\in \mathcal{Z}(B)$, and $r>0$.
Put
\[
\mathcal{A}=\{ X(h,r) : h\in \mathcal{Z}(A)\cup \mathcal{Z}(B) , \,
r>0\}.
\]
Then $\mathcal{A}$ is an $\mathcal{F}$-family on $X$ by Lemma
\ref{filters_fr-family}.
If $f_1,\ldots ,f_n\in \mathcal{Z}(A)$ and $g_1,\ldots ,g_m\in
\mathcal{Z}(B)$ for some $n,m\in \mathbb{N}$,
then $f:=\sum _{k=1}^n \lvert f_k\rvert \in \mathcal{Z}(A)$,
$g:=\sum _{k=1}^m \lvert g_k\rvert \in \mathcal{Z}(B)$.
If $r>0$, then 
\[
X(f,r)\cap X(g,r)\subseteq (\bigcap _{k=1}^n X(f_k,r) )\cap
  (\bigcap _{k=1}^m X(g_k,r)),
\] and so $\mathcal{A}$ has the finite
  intersection property.
By Lemma \ref{filters_existence}, there exists an element $p\in \delta
X$ such that $\mathcal{A}\subseteq p$.
Then $p\in \overline{e(A)}\cap \overline{e(B)}$ by Lemma
\ref{filters_closure}.
\end{proof}

Now, we are ready to prove first of the main theorems of this section.

\begin{theorem}
\label{filters_compact}
The space $\delta X$ is a compact Hausdorff space and $e(X)$
is dense in $\delta X$.
\end{theorem}

\begin{proof}
The density of $e(X)$ in $\delta X$ follows from Lemma
\ref{filters_properties} (\ref{filters_properties_iv}).
To see that $\delta X$ is Hausdorff, let $p$ and $q$ be
distinct points of $\delta X$.
Pick some set $A\in p\setminus q$.
Pick a set $B\in p$ and a function $f\in \mathcal{F}$  such that
$f(B)=\{0\}$ and $f(X\setminus A)=\{1\}$.
Since $X(f,1/2)\notin q$, there
exists a set $C\in q$ such that $X(f,1/3)\cap C=\emptyset $ by Theorem
\ref{filters_ultrafilter} (\ref{filters_ultrafilter_ii}).
Then $B\cap C=\emptyset $, and so
$\widehat{B}$ and $\widehat{C}$ are disjoint neighborhoods
of $p$ and $q$, respectively.

Let us show that $\delta X$ is compact.
Lemma \ref{filters_properties} (\ref{filters_properties_i}) implies
that the family
$\mathcal{B}=\{ \overline{e(A)} : A\subseteq X \}$ is a
base for the closed sets of $\delta X$.
Suppose that a subset $\mathcal{C}$ of $\mathcal{B}$
has the finite intersection property.
To show that $\delta X$ is compact,
it is enough to show that $\bigcap _{C\in \mathcal{C}}C\neq \emptyset $.
Put $\mathcal{A}'=\{ A\subseteq X : \overline{e(A)}\in
\mathcal{C} \}$ and, then, put $\mathcal{A}=\{ X(f,r) : A\in
\mathcal{A}', \, f\in \mathcal{Z}(A), \, r>0 \}$.
Now, $\mathcal{A}$ is an $\mathcal{F}$-family on $X$ by Lemma
\ref{filters_fr-family} and $\mathcal{A}$ has the finite intersection
property by Lemma \ref{filters_intersecting}.
By Lemma \ref{filters_existence},
there exists an element $p\in \delta X$ such that
$\mathcal{A}\subseteq p$.
Then $p\in \overline{e(A)}$ for every $A\in \mathcal{A}'$ by
Lemma \ref{filters_closure},
and so $p\in \bigcap _{C\in \mathcal{C}}C$,
as required.
\end{proof}

We finish this section by showing that $\mathcal{F}$-filters
describe the topology of $\delta X$.
As in the Stone-\v Cech compactification of a discrete topological
space, we have
two interpretations for the closure of an $\mathcal{F}$-filter in
$\delta X$, namely
$\widehat{\varphi }$ and the following.

\begin{definition}
Define $\overline{\varphi }=\bigcap _{A\in \varphi } \overline{e(A)}$
for every $\mathcal{F}$-filter $\varphi $ on $X$.
\end{definition}

Note that $\overline{\varphi }$ is a non-empty, closed subset of
$\delta X$.

\begin{theorem}
\label{filters_filtclos}
If $\varphi $ is an $\mathcal{F}$-filter on $X$, then $\widehat{\varphi
}=\overline{\varphi }$.
\end{theorem}

\begin{proof}
The inclusion $\widehat{\varphi }\subseteq \overline{\varphi }$
follows from Lemma \ref{filters_closure}.
To prove the reverse inclusion,
let $p\in \overline{\varphi }$ and
let $A\in \varphi $ with $A\neq X$.
Pick a set $B\in \varphi $ and a function
$f\in \mathcal{F}$ such that $f(B)=\{0\}$ and $f(X\setminus A)=\{1\}$.
Since $p\in \overline{e(B)}$,
we have $X(f,1/2)\in p$ by Lemma \ref{filters_closure},
and so $A\in p$.
Therefore, $\varphi \subseteq p$,
as required.
\end{proof}

\begin{theorem}
\label{filters_Cfilter}
If $C$ is a non-empty, closed subset of $\delta X$, then
there exists a unique $\mathcal{F}$-filter $\varphi $ on $X$ such that
$\widehat{\varphi }=C$.
\end{theorem}

\begin{proof}
Let $C$ be a non-empty, closed subset of $\delta X$.
Put $\varphi =\bigcap _{p\in C} p$.
Clearly, $\varphi $ is a filter on $X$.
Let us show that $\varphi $ is an $\mathcal{F}$-family on $X$, hence,
an $\mathcal{F}$-filter on $X$.
Let $A\in \varphi $.
If $p\in C$, then $A\in p$, and so there exist a set $B_p\in p$
and a function $f_p\in \mathcal{F}$ such that $f(X)\subseteq [0,1]$,
$f_p(B_p)=\{0\}$, and $f_p(X\setminus A)=\{1\}$.
Now, $\{ \widehat{B}_p : p\in C\}$ is an open cover of $C$,
and so there exist some $n\in \mathbb{N}$ and points $p_1,\ldots ,p_n\in C$
such that $C\subseteq \bigcup _{k=1}^n \widehat{B}_{p_k}$.
Put $f=\sum _{k=1}^n f_{p_k}$ and $B=\bigcup _{k=1}^n B_{p_k}$.
Then $B\in \varphi $ by Theorem \ref{filters_filtprop}.
Since $f(x)\leq n-1$ for every $x\in B$ and $f(x)=n$ for every
$x\in X\setminus A$,
the filter $\varphi $ is an $\mathcal{F}$-family on $X$.

Let us verify the equality $\widehat{\varphi }=C$.
The inclusion $C\subseteq \widehat{\varphi }$ is obvious,
so suppose that $q\in \delta X\setminus C$.
Then there exists a $\tau (\mathcal{F})$-open subset $A$ of $X$ such
that $A\in q$ and $\widehat{A}\cap C=\emptyset $.
For every $p\in C$, pick a $\tau (\mathcal{F})$-open subset $B_p$ of
$X$ such that $B_p\in p$ and $\widehat{A}\cap\widehat{B}_p=\emptyset $.
Then $A\cap B_p=\emptyset $ for every $p\in C$.
As above, there exist $n\in \mathbb{N}$ and points
$p_1,\ldots ,p_n\in C$ such that $B:=\bigcup _{k=1}^n B_{p_k}\in \varphi $.
Since $A\cap B=\emptyset $, we have $q\notin \widehat{\varphi }$, as
required.

Finally, uniqueness of $\varphi $ follows from Theorem
\ref{filters_filtprop} (\ref{filters_filtprop_iv}).
\end{proof}

\section{Continuous functions on $\delta X$}
\label{sec:cont-funct-delta}

This section is devoted to a study of continuous functions on the
space $\delta X$.
We show that the $C^{\ast }$-algebras $\mathcal{F}$ and $C(\delta X)$
are isometrically $\ast $-isomorphic, that is, we show that $\delta X$
is the spectrum of $\mathcal{F}$.
Furthermore, we show that every dense image of $X$ in a compact
Hausdorff space is determined up to homeomorphism by some
$C^{\ast }$-subalgebra of $\ell ^{\infty }(X)$ containing the constant
functions.

We leave the proof of the following simple lemma to the reader.

\begin{lemma}
\label{filters_rset}
If $p\in \delta X$, $g\in C(\delta X)$, and $r>0$,
then
\[
\{ x\in X : \lvert g(p)- g(e(x)) \rvert \leq r\} \in p.
\]
\end{lemma}

% \begin{proof}
% Let $r>0$ and put $U=\{ q\in \delta X : \lvert g(q)-g(p)\rvert \leq
% r\}$.
% Then there exists a
% $\tau (\mathcal{F})$-open subset $A$ of $X$ such that $A\in p$ and
% $\widehat{A}\subseteq U$.
% Since $A\subseteq e^{-1}(U)$, the statement follows from
% the equality
% $e^{-1}(U)=\{ x\in X : g(p)-r\leq g(e(x))\leq g(p)+r \}$.
% \end{proof}

\begin{theorem}
\label{filters_extension}
If $f\in \mathcal{F}$, then there exists a unique function
$\widehat{f}\in C(\delta X)$ with $f=\widehat{f}\circ
e$.
\end{theorem}

\begin{proof}
Let $p\in \delta X$ and define
\[
C=\bigcap _{A\in p} \text{cl}_{\mathbb{C}}(f(A)).
\]
Since $f$ is bounded and $p$ has the finite intersection property,
the set $C$ is a non-empty subset of $\mathbb{C}$.
Choosing any element $\widehat{f}(p)\in C$, we obtain a function
$\widehat{f}:\delta X\rightarrow \mathbb{C}$.

Next, let us show that if $p=\mathcal{N}_{\mathcal{F}}(x)$ for some
$x\in X$, then $C=\{f(x)\}$.
This will establish the equality $f=\widehat{f}\circ e$.
Clearly, $f(x)\in C$.
To see that $f(x)$ is the only member of $C$, let $y\in
\mathbb{C}$ with $y\neq f(x)$.
Pick $r>0$ such that $y\notin U:=\{ z\in \mathbb{C} : \lvert
f(x)-z\rvert \leq r\}$.
Then $f^{-1}(U)\in \mathcal{N}_{\mathcal{F}}(x)$.
Since $y\notin \text{cl}_{\mathbb{C}}(f(f^{-1}(U)))$, we have $y\notin
C$, as required.

Finally, let us show that $\widehat{f}$ is continuous.
This will also prove the uniqueness of $\widehat{f}$, since $e(X)$ is
dense in $\delta X$.
Let $p\in \delta X$ and put $g=f-\widehat{f}(p)$.
We claim that $X(g,r)\in p$ for every $r>0$.
By Corollary \ref{filters_extend}, it is enough to show that $X(g,r)\cap
B\neq \emptyset $ for every $B\in p$ and for every $r>0$.
So, let $B\in p$ and $r>0$ be given.
Since $\widehat{f}(p)\in \text{cl}_{\mathbb{C}}(f(B))$, there exists a point $x\in
B$ such that $\lvert g(x)\rvert =\lvert f(x)-\widehat{f}(p)\rvert \leq
r$, and so $x\in X(g,r)\cap B$, as required.
To finish the proof, let $r>0$.
If $q\in \widehat{X(g,r)}$, then $\widehat{f}(q)\in
\text{cl}_{\mathbb{C}}(f(X(g,r)))$,
and so $\lvert \widehat{f}(q)-\widehat{f}(p)\rvert \leq r$.
Therefore, $\widehat{f}$ is continuous at $p$.
\end{proof}

We defined the function $\widehat{f}$ by choosing any element from the
set $C$.
Since $\widehat{f}$ is unique, the set $C$ must be a
singleton for every $p\in \delta X$.
Although the evaluation mapping $e$ need not be injective, we call the
function $\widehat{f}$ an
\emph{extension} of $f$ to $\delta X$.

Following theorem can be deduced using
the Stone-Weierstrass Theorem.
However, we present the following proof using only properties of
$\mathcal{F}$-filters instead of the Stone-Weierstrass Theorem.

\begin{theorem}
\label{filters_isomorphic}
The mapping $\Gamma :\mathcal{F}\rightarrow C(\delta X)$
defined by $\Gamma (f)=\widehat{f}$ is an isometric $\ast $-isomorphism.
\end{theorem}

\begin{proof}
Using the density of $e(X)$ in $\delta X$ and the equality
$f=\widehat{f}\circ e$, it is easy to verify that $\Gamma $ is an
isometric $\ast $-homomorphism and we leave the details to the
reader.
To see that $\Gamma $ is surjective,
it is enough to show that for every positive function $g\in C(\delta
X)$ with $\lVert g\rVert =1$ and for every $r>0$, there
exists a function $f\in \mathcal{F}$ such that $\lVert
\widehat{f}-g\lVert \leq r$.
So, let $g\in C(\delta X)$ be positive
with $\lVert g\rVert =1$ and let $r>0$.
Pick $n\in \mathbb{N}$ such that $1/n\leq r/3$.
For every $k\in \{1,\ldots ,n\}$, define the following subsets of
$[0,1]$, $X$, and $\delta X$, respectively:
\begin{align*}
I_k&=\bigl[\frac{k-1}{n},\frac{k}{n}\bigr], \\
A_k&=\{ x\in X : \frac{k-2}{n} < g(e(x)) < \frac{k+1}{n} \}, \\
C_k&=g^{-1}(I_k).
\end{align*}
Note that $A_k\cap A_j=\emptyset $ whenever $k,j\in
\{1,\ldots ,n\}$ and $k+3\leq j$.

Let $k\in \{1,\ldots ,n\}$.
If $p\in C_k$, then $A_k\in p$ by Lemma \ref{filters_rset},
and so there exist a set $B_p\in p$ and a positive
function $f_p\in \mathcal{F}$ such that
$f_p(B_p)=\{k/n\}$, $f_p(X\setminus A_k)=\{0\}$, and
$f_p(X)\subseteq [0, k/n]$.
Pick $p_1,\ldots ,p_n\in C_k$ for some $n\in \mathbb{N}$ such
that $C_k\subseteq \bigcup _{j=1}^n \widehat{B}_{p_j}$ and
put $f_k=f_{p_1}\vee \ldots \vee f_{p_{n}}$.
Note that $f_k(X\setminus A_k)=\{0\}$ and $f_k(x)=k/n$ for every
$x\in X$ with $e(x)\in C_k$.

Put $f=f_1\vee \ldots \vee f_n$.
We claim that $\lVert \widehat{f}-g\rVert \leq r$.
To prove this, it is enough to show that $\lvert
f(x)-g(e(x))\rvert \leq r$ for every $x\in X$.
So, let $x\in X$.
If $g(e(x))\geq (n-3)/n$, then $e(x)\in C_k$ for some $k\in
\mathbb{N}$ with $n-2\leq k\leq n$,
and so $f(x)\geq (n-2)/n$.
Therefore, $\lvert f(x)-g(e(x))\rvert \leq r$.
On the other hand, if $g(e(x))<(n-3)/n$,
then there exists an element $k\in \{1,\ldots ,n-3\}$ such that $(k-1)/n\leq g(e(x))<k/n$.
Now, $x\in A_k$ and $e(x)\in C_k$, and so $f(x)\geq k/n$.
Since $A_k\cap A_j=\emptyset $ for every $j\in \{1,\ldots ,n\}$ with
$j\geq k+3$, we have $f_j(x)=0$ for every $k+3\leq j\leq n$,
and so $f(x)\leq (k+2)/n$.
Therefore, $\lvert f(x)-g(e(x))\rvert \leq 3/n\leq r$, thus finishing
the proof.
\end{proof}

Next, we show that $\mathcal{F}$-filters describe
all dense images of $X$ in compact Hausdorff spaces.
Precise statement and details follow.

\begin{theorem}
\label{filters_denseim}
Suppose that $Y$ is a compact Hausdorff space and $\varepsilon
:X\rightarrow Y$ is a function such that $\varepsilon (X)$ is dense in
$Y$.
The following statements hold:
\begin{enumerate}[\upshape (i)]
\item\label{filters_denseim_i}
The set $\mathcal{F}=\{ h\circ \varepsilon : h\in C(Y) \}$
is a $C^{\ast }$-subalgebra of $\ell ^{\infty }(X)$ containing the constant functions.
\item\label{filters_denseim_ii}
$\mathcal{F}$ is isometrically isomorphic with $C(Y)$.
\item\label{filters_denseim_iii}
There exists a homeomorphism $F:\delta X\rightarrow Y$
such that $F\circ e=\varepsilon $.
\end{enumerate}
\end{theorem}

\begin{proof}
We only prove statement (\ref{filters_denseim_iii}) and leave the easy verifications of
(\ref{filters_denseim_i}) and (\ref{filters_denseim_ii}) to the reader.
If $p\in \delta X$, then
\[
C=\bigcap _{A\in p} \text{cl}_Y(\varepsilon (A))
\]
is a non-empty subset of $Y$.
Suppose that $x,y\in C$ with $x\neq y$.
By Urysohn's Lemma, there exists a real-valued function $h\in C(Y)$
with $h(x)=0$ and $h(y)=1$.
Put $f=h\circ \varepsilon $
and $A=\{ x\in X : \widehat{f}(p)-1/3\leq
f(x)\leq \widehat{f}(p)+1/3 \}$.
Now, $A\in p$ by Lemma \ref{filters_rset}, so $x,y\in
\text{cl}_Y(\varepsilon (A))$, and so
$h(x),h(y)\in \text{cl}_{\mathbb{R}} (f(A))$.
Therefore, $\lvert h(x)-h(y)\rvert \leq 2/3$, a contradiction.

Since $C$ is a singleton, we obtain a function $F:\delta
X\rightarrow Y$.
Clearly, $F\circ e=\varepsilon $.
Since $e(X)$ and $\varepsilon (X)$ are dense in $\delta X$ and $Y$,
respectively, we need only to show that $F$ is injective and
continuous to finish the proof.

Let $p,q\in \delta X$ with $p\neq q$.
By Urysohn's Lemma, there exists a real-valued function
$g\in C(\delta X)$ such that $g(p)=0$ and $g(q)=1$.
Put $f=g\circ e$.
Then $f\in \mathcal{F}$ by Theorem \ref{filters_isomorphic}.
Put $A=\{ x\in X : f(x)\leq 1/3 \}$
and $B=\{ x\in X : f(x)\geq 2/3\}$.
Then $A\in p$ and $B\in q$ by Lemma \ref{filters_rset},
and so $F(p)\in \text{cl}_Y(\varepsilon (A))$ and $F(q)\in
\text{cl}_Y(\varepsilon (B))$.
By statement (\ref{filters_denseim_ii}), there exists a function $h\in C(Y)$
such that $f=h\circ \varepsilon $.
Then $h(F(p))\in \text{cl}_{\mathbb{R}}(f(A))$ and
$h(F(q))\in \text{cl}_{\mathbb{R}}(f(B))$,
and so $F(p)\neq F(q)$, as required.

To finish the proof, let $p\in \delta X$ and let $U$
be an open neighborhood of $F(p)$ in $Y$ with $U\neq Y$.
Again, there exists a real-valued function $h\in C(Y)$
such that $h(\widehat{f}(p))=0$ and $h(Y\setminus U)=\{1\}$.
Put $f=h\circ \varepsilon $.
Using the continuity of $h$ as above, we obtain $\widehat{f}(p)=0$,
and so $B=\{ x\in X : -1/2\leq f(x)\leq 1/2 \} \in p$ by Lemma
\ref{filters_rset}.
If $q\in \widehat{B}$, then
$h(F(q))\in [-1/2,1/2]$, and so $F(q)\in U$.
Therefore, $F$ is continuous at $p$.
\end{proof}

\section{Some relationships between $C^{\ast }$-subalgebras of $\ell
  ^{\infty }(X)$}
\label{sec:inclusions-etc}

Throughout this section, we assume that $\mathcal{F}_1$ and
$\mathcal{F}_2$ are $C^{\ast }$-subalgebras of $\ell ^{\infty }(X)$
containing the constant functions.
We denote by $\delta _1X$, and $\delta _2X$ the spaces of
$\mathcal{F}_1$-ultrafilters and $\mathcal{F}_2$-ultrafilters on $X$,
respectively.
Also, we denote by $e_1$ and $e_2$ the evaluation mappings from $X$ to
$\delta _1X$ and $\delta _2X$, respectively.
If $A\subseteq X$, then the notation $\widehat{A}$ is ambiguous.
However, it should be clear from the context whether we consider
$\widehat{A}$ as a subset of $\delta _1X$ or $\delta _2X$.
If $f\in \mathcal{F}_1\cap \mathcal{F}_2$,
then $f$ extends to both $\delta _1X$ and $\delta _2X$.
We denote these extension by $f^{\delta _1}$ and $f^{\delta _2}$,
respectively.

\begin{theorem}
\label{filters_Fincl}
The inclusion $\mathcal{F}_1\subseteq \mathcal{F}_2$ holds if and only
if  there exists a continuous, surjective mapping
$F:\delta _2X\rightarrow \delta _1X$ such that $e_1=F\circ e_2$.
\end{theorem}

\begin{proof}
Suppose first that $\mathcal{F}_1\subseteq \mathcal{F}_2$.
Let $p\in \delta _2X$ and put
\[
C=\bigcap _{A\in p} \text{cl}_{\delta _1X} (e_1(A)).
\]
Similar arguments as in the proof of Theorem \ref{filters_denseim}
apply to show that $C$ is a singleton,
and so we obtain a function $F:\delta _2X\rightarrow \delta _1X$.
Clearly, $e_1=F\circ e_2$.
Also, arguing as in the last part of the proof of Theorem
\ref{filters_denseim} we see that $F$ is continuous.
Therefore, we need only to show that $F$ is surjective.

If $q\in \delta _1X$, then $q$ is an $\mathcal{F}_2$-filter on $X$.
Pick any $p\in \delta _2X$ such that $q\subseteq p$.
Let $A\in q$.
Since $\delta _1X$ is a regular topological space, there exists a $\tau
(\mathcal{F}_1)$-open subset $B$ of $X$ such that $B\in q$ and
$\text{cl}_{\delta  _1X}(\widehat{B})\subseteq \widehat{A}$.
Now, $B\in p$, so $F(p)\in \text{cl}_{\delta _1X}(\widehat{B})$ by Lemma
\ref{filters_properties} (\ref{filters_properties_ii}), and so $A\in F(p)$.
Therefore, $q\subseteq F(p)$, and so $q=F(p)$, as required.

Suppose now that there exists a mapping $F:\delta _2X\rightarrow \delta
_1X$ as above.
Let $f\in \mathcal{F}_1$.
By Theorem \ref{filters_extension}, there exists a function
$g\in C(\delta _1X)$ with $f=g\circ e_1$,
and so $f=(g\circ F)\circ e_2$.
Since $g\circ F\in C(\delta _2X)$,
we have $f\in \mathcal{F}_2$ by Corollary \ref{filters_isomorphic}.
\end{proof}

For the proof of the next theorem,
recall the definition of $\widehat{f}$ from the
proof of Theorem \ref{filters_extension}.

\begin{theorem}
\label{filters_surjpts}
Suppose that $\mathcal{F}_1\subseteq \mathcal{F}_2$ and
let $F:\delta _2X\rightarrow \delta _1X$ be as in Theorem
\ref{filters_Fincl}.
If $p\in \delta _2X$ and $q\in \delta _1X$, then the following
statements are equivalent:
\begin{enumerate}[\upshape (i)]
\item\label{filters_surjpts_i}
$q\subseteq p$.
\item\label{filters_surjpts_ii}
$F(p)=q$.
\item\label{filters_surjpts_iii}
$f^{\delta _2}(p)=f^{\delta _1}(q)$ for every $f\in \mathcal{F}_1$.
\end{enumerate}
\end{theorem}

\begin{proof}
(i) $\Rightarrow $ (ii)
This was already proved in the proof of Theorem
\ref{filters_Fincl}.

(ii) $\Rightarrow $ (iii)
Suppose that $F(p)=q$.
Let $f\in \mathcal{F}_1$.
Since $e_1=F\circ e_2$, the functions $f^{\delta _2}$ and $f^{\delta
  _1}\circ F$ agree on $e_2(X)$.
Therefore, $f^{\delta _2}=f^{\delta _1}\circ F$, and so
$f^{\delta _2}(p) =f^{\delta _1}(q)$.

(iii) $\Rightarrow $ (i)
Suppose that $q$ is not contained in $p$.
Pick some set $A\in q\setminus p$.
Pick a set $B\in q$ and a positive function $f\in
\mathcal{F}_1$ such that $f(B)=\{0\}$ and $f(X\setminus A)=\{1\}$.
Then $f^{\delta _1}(q)=0$.
Since $X(f,1/2)\notin p$, there exists a set $C\in p$ such that
$X(f,1/3)\cap C=\emptyset $ by Theorem \ref{filters_ultrafilter}.
Then $f^{\delta _2}(p)\geq 1/3$, thus finishing the proof.
\end{proof}

Define two closed equivalence relations $\sim $ and $\approx $ on
$\delta _2X$ as follows:
$p\sim q$ if and only if $F(p)=F(q)$, and
$p\approx q$ if and only if $f^{\delta _2}(p)=f^{\delta _2}(q)$ for
every $f\in \mathcal{F}_1$.
Theorem \ref{filters_surjpts} shows that these relations are identical.
Since $F$ is a quotient mapping (see \cite[pp. 60-61]{Willard}),
we obtain the following statement.

\begin{corollary}
If $\mathcal{F}_1\subseteq \mathcal{F}_2$,
then the quotient space $\delta _2X/\approx $ is homeomorphic with
$\delta _1X$.
\end{corollary}

% \begin{theorem}
% \label{filters_F1dense}
% Suppose that $\mathcal{F}_1$ is a conjugate closed subalgebra of $\ell
% ^{\infty }(X)$ such that $\mathcal{F}_1\subseteq \mathcal{F}_2$.
% The following statements are equivalent:
% \begin{enumerate}[4]
% \item\label{filters_F1dense_i}
% $\mathcal{F}_1$ is dense in $\mathcal{F}_2$.
% \item\label{filters_F1dense_ii}
% The set $\{ f^{\delta _2} : f\in \mathcal{F}_1 \}$
%   is dense in $C(\delta _2X)$.
% \item\label{filters_F1dense_iii}
% A filter $\varphi $ on $X$ is an $\mathcal{F}_1$-filter if and
%   only if $\varphi $ is an $\mathcal{F}_2$-filter.
% \item\label{filters_F1dense_iv}
% $\delta _1X=\delta _2X$.
% \end{enumerate}
% \end{theorem}

% \begin{proof}
% The implications (\ref{filters_F1dense_i}) $\Rightarrow $
% (\ref{filters_F1dense_ii}), (\ref{filters_F1dense_ii}) $\Rightarrow $
% (\ref{filters_F1dense_i}), and
% (\ref{filters_F1dense_iv}) $\Rightarrow $ (\ref{filters_F1dense_i})
% follow from Theorem \ref{filters_isomorphic}
% and the implication (\ref{filters_F1dense_iii}) $\Rightarrow $
% (\ref{filters_F1dense_iv}) is obvious.
% To establish the implication (\ref{filters_F1dense_i}) $\Rightarrow $ (\ref{filters_F1dense_iii}),
% it is enough to show that a non-empty subset $\mathcal{A}$ of
% non-empty subsets of $X$ is an $\mathcal{F}_1$-family on $X$ if and
% only if $\mathcal{A}$ is an $\mathcal{F}_2$-family on $X$.
% Since $\mathcal{F}_1\subseteq \mathcal{F}_2$, necessity is obvious,
% and sufficiency follows from the density of $\mathcal{F}_1$ in
% $\mathcal{F}_2$ and Remark \ref{filters_Ffamrem}.
% \end{proof}

\section{$\mathcal{F}$-filters and ideals of $\mathcal{F}$}
\label{sec:mathc-filt-ideals}

In this section, we establish a correspondence between
$\mathcal{F}$-filters on $X$ and ideals of $\mathcal{F}$.
Roughly speaking, we show how the ideals of $\mathcal{F}$ can be used
to generate $\mathcal{F}$-filters on $X$.
We apply the following convention for the rest of the paper:
\emph{By an ideal of $\mathcal{F}$, we always mean a closed,
  proper ideal of $\mathcal{F}$}.

The next lemma follows from \cite[(1.23) Proposition]{Folland}.
Since the proof of the cited proposition relies on the spectrums of
single elements of $C^{\ast }$-algebras,
we present the following short proof using only basic properties of
Banach algebras.

\begin{lemma}
\label{filters_inverse}
If $f\in \mathcal{F}\setminus \mathcal{F}_0$, then $1/f\in \mathcal{F}
$.
\end{lemma}

\begin{proof}
Suppose first that $f\in \mathcal{F}\setminus \mathcal{F}_0$ is
positive.
Pick $r>0$ such that $r\leq f(x)$ for every $x\in X$.
Then
\[
0<\frac{r}{\lVert f\rVert }\leq \frac{f(x)}{\lVert f\rVert }\leq 1
\]
for every $x\in X$.
Put $g=f/\lVert f\rVert $.
Then $\lVert 1-g\rVert <1$ by the inequalities above,
and so $g$ is invertible in $\mathcal{F} $ (see \cite[(1.3) Lemma]{Folland}).
Therefore, $1/f\in \mathcal{F} $.

If $f\in \mathcal{F}\setminus \mathcal{F}_0$ is any function,
then $\lvert f\rvert ^2\in \mathcal{F}\setminus
\mathcal{F}_0$ is positive.
The equality $1/f=\overline{f}/\lvert f\rvert ^2$
and the first part of the proof imply that $1/f\in \mathcal{F} $.
\end{proof}

\begin{corollary}
\label{filters_inf0}
If $I$ is an ideal of $\mathcal{F} $, then $I\subseteq \mathcal{F} _0$.
\end{corollary}

\begin{definition}
For every ideal $I$ of $\mathcal{F}$,
define
\[
\mathcal{B}(I)=\{ X(f,r) : f\in I , \, r>0 \}.
\]
\end{definition}

\begin{theorem}
\label{filters_filtcond}
If $\varphi $ is an $\mathcal{F}$-filter on $X$,
then there exists an ideal $I$ of $\mathcal{F}$ such that $\varphi
$ is generated by $\mathcal{B}(I)$.
Conversely, if $I$ is an ideal of $\mathcal{F}$, then $\mathcal{B}(I)$
is a filter base on $X$ and the filter $\varphi $ on $X$ generated by
$\mathcal{B}(I)$ is an $\mathcal{F}$-filter.
\end{theorem}

\begin{proof}
Suppose first that $\varphi $ is an $\mathcal{F}$-filter on $X$.
Put
\[
I=\{ f\in \mathcal{F} : X(f,r)\in \varphi \text{ for every } r>0
\}.
\]
Clearly, $0\in I$.
Let $f_1,f_2\in I$, let $h\in \mathcal{F}$ with $h\neq 0$, let
$\alpha \in \mathbb{C}$ with $\alpha \neq 0$, let $(g_n)$
be a sequence in $I$ which converges to some $g\in \mathcal{F}$,
and let $r>0$.
The inclusions
\begin{align*}
X(f_1,r/2)\cap X(f_2,r/2)&\subseteq X(f_1-f_2,r), \\
X(f_1,r/\lVert h\rVert )&\subseteq X(f_1h,r), \\
X(f_1,r/\lvert \alpha \rvert )& \subseteq X(\alpha f_1,r), \\
X(g_n,r/2) & \subseteq X(g,r),
\end{align*}
where the last one holds if $\lVert g_n-g\rVert \leq r/2$,
imply that $I$ is an ideal of $\mathcal{F}$.

We claim that $\mathcal{B}(I)$ is a filter base for $\varphi $.
Clearly, $\mathcal{B}(I)\subseteq \varphi $, so
let $A\in \varphi $ with $A\neq X$.
Pick a set $B\in \varphi $ and a function $f\in \mathcal{F}$ with
$f(B)=\{0\}$ and $f(X\setminus A)=\{1\}$.
Since $B\subseteq X(f,r)$ for every $r>0$, we have $f\in I$.
Since $X(f,1/2)\subseteq A$, the claim follows.

Suppose now that $I$ is an ideal of $\mathcal{F}$.
First, $X(f,r)\neq \emptyset $ for every $f\in I$ and for every $r>0
$ by Corollary \ref{filters_inf0}.
Next, let $f,g\in I$ and let $r>0$.
Since $\lvert f\rvert ^2+\lvert g\rvert ^2\in I$ and
$X(\lvert f\rvert ^2+\lvert g\rvert ^2,r)\subseteq X(f,r)\cap X(g,r)$,
the set $\mathcal{B}(I)$ is a filter base on $X$.
Since $\mathcal{B}(I)$ is an $\mathcal{F}$-family on $X$ by Lemma
\ref{filters_fr-family}, the filter $\varphi $ on $X$ generated by
$\mathcal{B}(I)$ is an $\mathcal{F}$-filter.
\end{proof}

\begin{theorem}
\label{filters_3cond}
Let $I$ be an ideal of $\mathcal{F}$,
let $\varphi $ be the $\mathcal{F}$-filter on $X$
generated by $\mathcal{B}(I)$,
and let $f\in \mathcal{F}$.
The following statements are equivalent:
\begin{enumerate}[\upshape (i)]
\item\label{filters_3cond_i}
$f\in I$.
\item\label{filters_3cond_ii}
$\widehat{f}(p)=0$ for every $p\in \overline{\varphi }$.
\item\label{filters_3cond_iii}
$X(f,r)\in \varphi $ for every $r>0$.
\end{enumerate}
\end{theorem}

\begin{proof}
(i) $\Rightarrow $ (ii)
Suppose that $f\in I$.
Let $p\in \overline{\varphi }$ and let $r>0$.
Since $\varphi $ is generated by $\mathcal{B}(I)$, we have $X(f,r)\in p$,
and so $\lvert \widehat{f}(p)\rvert \leq r$ by Lemma
\ref{filters_closure}.
Therefore, $\widehat{f}(p)=0$.

(ii) $\Rightarrow $ (iii)
This follows from Lemma \ref{filters_rset} and Theorem
\ref{filters_filtprop} (\ref{filters_filtprop_ii}).

(iii) $\Rightarrow $ (i)
Suppose that (\ref{filters_3cond_iii}) holds.
We need only to show that $f\in \text{cl}_{\mathcal{F}}(I)$, and so
we may assume that $f\neq 0$.
Let $0<r<\lVert f\rVert $.
Then $X(f,r)\neq X$.
Since $\varphi $ is generated by $\mathcal{B}(I)$,
there exist some functions $h\in I$ and $g\in \mathcal{F}$ such that
$g(X(h,1))=\{0\}$ and $g(X\setminus X(f,r))=\{1\}$.
Now,
$1/(\lvert h\rvert
\vee 1)^2\in \mathcal{F}$ by Lemma \ref{filters_inverse},
so $k:=\lvert h\rvert ^2/(\lvert h\rvert \vee 1)^2\in I$,
and so $fk\in I$.
Note that $k(X)\subseteq [0,1]$.
Since $X(h,1)\subseteq X(f,r)$, the functions $f$ and $fk$ agree on
$X\setminus X(f,r)$.
Therefore,
$\lVert f-fk\rVert =\sup _{x\in X(f,r)} \lvert f(x)(1-k(x)) \rvert
\leq r$, and so $f\in \text{cl}_{\mathcal{F}}(I)$, as required.
\end{proof}

Let $I$ be an ideal of $\mathcal{F}$.
The equalities $X(f,r)=X(\lvert f\rvert ,r)=X(\overline{f},r)$ for every $f\in
\mathcal{F}$ and for every $r>0$ imply that $\lvert
f\rvert \in I$ and $\overline{f}\in I$ for every $f\in I$.

Theorems \ref{filters_filtcond} and \ref{filters_3cond} imply that,
for every $\mathcal{F}$-filter $\varphi $ on $X$,
there exists a unique ideal $I$ of $\mathcal{F}$ such that
$\varphi $ is generated by $\mathcal{B}(I)$.
If $I$ and $J$ are ideals of $\mathcal{F}$
such that $I\subseteq J$,
then $\mathcal{B}(I)\subseteq \mathcal{B}(J)$.
From this we conclude the following:
If $I$ is an ideal of $\mathcal{F}$ and $\varphi $ is the
$\mathcal{F}$-filter on $X$ generated by $\mathcal{B}(I)$,
then $I$ is a maximal ideal of $\mathcal{F}$ if and only if
$\varphi $ is an $\mathcal{F}$-ultrafilter on $X$.
The following well-known property of $\mathcal{F}$
follows from Theorem
\ref{filters_filtprop} (\ref{filters_filtprop_ii}).

\begin{corollary}
If $I$ is an ideal of $\mathcal{F}$, then $I$ is the intersection of
all of those maximal ideals of $\mathcal{F}$ which contain $I$.
\end{corollary}

% \begin{proof}
% Put $h(I)=\{ J\subseteq \mathcal{F} : J \text{ is a maximal
%   ideal of $\mathcal{F}$ such that $I\subseteq J$ } \}$ and, then, put
% $k(h(I))=\bigcap _{J\in h(I)} J$.
% The inclusion $I\subseteq k(h(I))$ is obvious, so let $f\in k(h(I))$.
% Let $\varphi $ be the $\mathcal{F}$-filter on $X$ generated by
% $\mathcal{B}(I)$ and,
% for every $J\in h(I)$, let $p_J$ be the $\mathcal{F}$-ultrafilter on
% $X$ generated by $\mathcal{B}(J)$.
% If $p\in \delta X$, then $\varphi \subseteq p$ if and only if $p=p_J$
% for some $J\in h(I)$.
% Therefore, $\varphi =\bigcap _{J\in h(I)} p_J$ by Theorem
% \ref{filters_filtprop} (ii).
% Since $f\in J$ for every $J\in h(I)$, we have $X(f,r)\in p_J$ for
% every $r>0$ and for every $J\in h(I)$ by Theorem \ref{filters_3cond}, so
% $X(f,r)\in \varphi $ for every $r>0$, and so $f\in I$ by Theorem
% \ref{filters_3cond}.
% \end{proof}

\begin{remark}
\label{filters_maxideal}
Let $\Delta $ be the spectrum of $\mathcal{F}$.
We consider $\Delta $ as the space of all non-zero, multiplicative linear
functionals on $\mathcal{F}$.
The \emph{evaluation mapping} $\varepsilon :X\rightarrow \Delta $ is
defined by $[\varepsilon (x)](f)=f(x)$ for every $x\in X$ and for
every $f\in \mathcal{F}$.
If $x\in X$, then the $\mathcal{F}$-filter on $X$ generated by
$\mathcal{B}(\ker \varepsilon (x))$ is $\mathcal{N}_{\mathcal{F}}(x)$
by Lemma \ref{filters_Nxbase},
By Theorem \ref{filters_denseim},
the mapping $\mu \mapsto p_{\mu }$ from $\Delta $
to $\delta X$,
where $p_{\mu }$ is the
$\mathcal{F}$-ultrafilter on $X$ generated by $\mathcal{B}(\ker \mu )$,
is a homeomorphism.
\end{remark}

\section{$\mathcal{F}$-filters on topological spaces}
\label{sec:mathc-filt-topol}

In the previous sections,
we made no assumption about any kind of structure on the set $X$.
In this section, we assume that $(X,\tau )$ is a Hausdorff topological space and
$\mathcal{F}\subseteq C(X)$.

Recall that $A^{\circ }$ denotes the $\tau (\mathcal{F})$-interior of
a subset $A$ of $X$.
If $A\subseteq X$, then $e^{-1}(\widehat{A}) = A^{\circ }$.
Since $\mathcal{F}\subseteq C(X)$, the set $A^{\circ }$
is open in $X$, and so the evaluation mapping $e:X\rightarrow \delta
X$ is continuous.
For every element $x\in X$,  we denote by $\mathcal{N}(x)$ the
neighborhood filter of $x$ in $(X,\tau )$.
Since $\mathcal{F}\subseteq C(X)$, we have
$\mathcal{N}_{\mathcal{F}}(x)\subseteq \mathcal{N}(x)$ for every $x\in
X$.

\emph{For the rest of this section, all the topological properties on
  $X$ or on its subsets are taken with respect to the original topology
  $\tau $ of $X$.}

Next theorem follows from Theorem \ref{filters_denseim}.

\begin{theorem}
\label{filters_contdenseim}
If $Y$ is a compact Hausdorff space and $\varepsilon:X\rightarrow Y$
is a continuous mapping such that $\varepsilon (X)$ is dense in
$Y$,
then the following statements hold:
\begin{enumerate}[(i)]
\item\label{filters_contdenseim_i}
The set $\mathcal{F}=\{ h\circ \varepsilon : h\in C(Y) \}$
is a $C^{\ast }$-subalgebra of $C(X)$ containing the constant
functions.
\item\label{filters_contdenseim_ii}
$\mathcal{F}$ is isometrically isomorphic to $C(Y)$.
\item\label{filters_contdenseim_iii}
There exists a homeomorphism $F:\delta X\rightarrow Y$
such that $F\circ e=\varepsilon $.
\end{enumerate}
\end{theorem}

The evaluation mapping $e:X\rightarrow \delta X$ is an embedding if
and only if the equality $\mathcal{N}(x)=\mathcal{N}_{\mathcal{F}}(x)$
holds for every $x\in X$.
By Remark \ref{filters_fr-family}, the latter statement is equivalent
to statement (\ref{filters_embedding_ii}) below.

\begin{lemma}
\label{filters_embedding}
The following statements are equivalent:
\begin{enumerate}[\upshape (i)]
\item The canonical mapping $e:X\rightarrow \delta X$ is an embedding.
\item\label{filters_embedding_ii}
For every element $x\in X$ and for every neighborhood
$A\in \mathcal{N}(x)$ with $A\neq X$,
there exists a function $f\in \mathcal{F}$ with
$f(x)=1$ and $f(X\setminus A)=\{0\}$.
\end{enumerate}
\end{lemma}

Statements (\ref{filters_evalcomp_ii}) and
(\ref{filters_evalcomp_iii}) of the next corollary constitute the
Stone-Weierstrass Theorem.

\begin{corollary}
\label{filters_evalcomp}
Suppose that $X$ is compact and that $\mathcal{F}'$ is a conjugate
closed subalgebra of $C(X)$.
Let $\mathcal{F}$ be the closure of $\mathcal{F}'$ in $C(X)$.
The following statements are equivalent:
\begin{enumerate}[\upshape (i)]
\item\label{filters_evalcomp_i}
The evaluation mapping $e:X\rightarrow \delta X$ is a homeomorphism.
\item\label{filters_evalcomp_ii}
$\mathcal{F}'$ separates the points of $X$.
\item\label{filters_evalcomp_iii}
$\mathcal{F}=C(X)$, that is,
$\mathcal{F}'$ is dense in $C(X)$.
\end{enumerate}
\end{corollary}

\begin{proof}
Since $X$ is compact, the evaluation mapping $e:X\rightarrow \delta X$
is a continuous surjection.
Therefore, $e$ is a homeomorphism if and only if $e$ is
injective.
Since $\mathcal{F}'$ is dense in $\mathcal{F}$, statements (\ref{filters_evalcomp_i}) and
(\ref{filters_evalcomp_ii}) are equivalent.

(i) $\Rightarrow $ (iii)
Suppose that $e:X\rightarrow \delta X$ is a homeomorphism.
Then it is easy to
verify that the mapping $g\mapsto g\circ e$ from $C(\delta X)$ to
$C(X)$ is an isometric $\ast $-isomorphism.
Since $\mathcal{F}$ is isometrically $\ast $-isomorphic with $C(\delta
X)$ by Theorem \ref{filters_isomorphic}, the statement follows.

(iii) $\Rightarrow $ (ii)
This follows from Urysohn's Lemma.
\end{proof}

Next statement is a consequence of the Gelfand-Naimark Theorem.
Here, it follows from Remark \ref{filters_maxideal} and
Corollary \ref{filters_evalcomp}.

\begin{corollary}
If $X$ and $Y$ are compact Hausdorff spaces, then $X$ and $Y$ are
homeomorphic if and only if $C(X)$ and $C(Y)$
are isometrically $\ast $-isomorphic.
\end{corollary}

% \begin{proof}
% Necessity follows as in the proof of Corollary \ref{filters_evalcomp}
% and sufficiency follows from Remark \ref{filters_maxideal} and
% Corollary \ref{filters_evalcomp}.
% \end{proof}

We finish the paper with the following statement concerning locally
compact spaces.
If $X$ is locally compact,
we denote by $X_{\infty }$ the one-point compactification of $X$.
Let $e_1:X\rightarrow X_{\infty }$ denote the natural embedding.
Then it is
easy to verify that $\{ h\circ e_1 : h\in
C(X_{\infty }) \} =C_0(X)\oplus \mathbb{C}$,
where $\mathbb{C}$ denotes the constant functions on $X$.

\begin{theorem}
\label{filters_embopen}
If $X$ is non-compact and locally compact,
then the following statements are equivalent:
\begin{enumerate}[\upshape (i)]
\item\label{filters_embopen_i}
The evaluation mapping $e:X\rightarrow \delta X$ is
an embedding and $e(X)$ is open in $\delta X$.
\item $C_0(X)\subseteq \mathcal{F} $.
\item\label{filters_embopen_iii}
There exists a continuous surjection $F:\delta X\rightarrow
  X_{\infty }$ such that $F(e(x))=x$ for every $x\in X$.
\item\label{filters_embopen_iv}
The set $\varphi =\{ X\setminus K : K\subseteq X \text{ and }
  \text{cl}_X(K) \text{ is compact} \}$ is an $\mathcal{F}$-filter on
  $X$ and $\mathcal{F}$ separates the points of $X$.
\end{enumerate}
\end{theorem}

\begin{proof}
(i) $\Rightarrow $ (ii)
Suppose that (\ref{filters_embopen_i}) holds.
Let $f\in C_0(X)$.
Define
$F:\delta X\rightarrow \mathbb{C} $ by 
\[
F(p)=\begin{cases}
f(x) & \text{if } p=e(x) \text{ for some } x\in X, \\
0 & \text{otherwise}.
\end{cases}
\]
It is easy to verify that $F$ is continuous, and so $f\in \mathcal{F}$
by Corollary \ref{filters_isomorphic}.

(ii) $\Rightarrow $ (iii)
This follows from Theorems \ref{filters_Fincl} and
\ref{filters_contdenseim} and the remark preceding this theorem.

(iii) $\Rightarrow $ (iv)
Since $X$ is not compact, the set $\varphi $ is a filter on $X$.
Suppose now that (\ref{filters_embopen_iii}) holds.
Clearly, $\mathcal{F}$ separates the points of $X$.
To see that $\varphi $ is an $\mathcal{F}$-filter on $X$,
let $A\in \varphi $.
Pick a subset $K$ of $X$ such that $A=X\setminus K$ and
$\text{cl}_X(K)$ is compact.
Pick an open subset $U$ of $X$ such that $K\subseteq U$ and
$\text{cl}_X(U)$ is compact.
By Urysohn's Lemma, there exists a function $h\in
C(X_{\infty })$ such that $h(\text{cl}_X(K))=\{1\}$ and
$h(X_{\infty }\setminus U)=\{0\}$.
By assumption and Corollary \ref{filters_isomorphic},
there exists a function $f\in \mathcal{F}$ such that
$\widehat{f}=h\circ F$.
Since $F(e(x))=x$ for every $x\in X$,
we have $f(\text{cl}_X(K))=\{1\}$ and $f(X\setminus U)=\{0\}$.
Since $X\setminus A\subseteq \text{cl}_X(K)$ and $X\setminus U\in
\varphi $, the statement follows.

(iv) $\Rightarrow $ (i)
Suppose that (\ref{filters_embopen_iv}) holds.
By Theorem \ref{filters_filtcond}, there exists an ideal $I$ of
$\mathcal{F}$ such that $\varphi $ is generated by $\mathcal{B}(I)$.
If $f\in I$, then $X(f,r)\in \varphi $ for every $r>0$,
and so we must have $f\in C_0(X)$.
On the other hand, if $f\in C_0(X)$, then $X(f,r)\in \varphi
$ for every $r>0$.
The proof of implication (\ref{filters_3cond_iii}) $\Rightarrow $
(\ref{filters_3cond_i}) in the proof of Theorem \ref{filters_3cond}
applies to show that $f\in I$.
Therefore, $I=C_0(X)$, and so the evaluation mapping is an
embedding by Lemma \ref{filters_embedding} (see \cite[p. 85]{Alipr}).

To finish the proof, it is enough to show that $\delta X\setminus
e(X)=\widehat{\varphi }$.
Since $X$ is locally compact,
the inclusion $\widehat{\varphi }\subseteq \delta X\setminus e(X)$ is
obvious.
To verify the reverse inclusion, let $p\in \delta X\setminus e(X)$
and let $K$ be a compact subset of $X$.
Since $e$ is continuous, we have $p\in \delta X\setminus \overline{e(K)}$,
and so $X\setminus K\in p$ by Lemma \ref{filters_properties} (\ref{filters_properties_i}).
Therefore, $\varphi \subseteq p$, as required.
\end{proof}

%\begin{acknowledgements}
%\end{acknowledgements}



\end{document}